\documentclass[11pt]{amsart}
\usepackage{amsmath,amsthm,amsfonts,amscd,amssymb,eucal,latexsym,mathrsfs,stmaryrd,enumitem}

\newtheorem{theorem}{Theorem}[section]

\newtheorem{lemma}[theorem]{Lemma}
\newtheorem{proposition}[theorem]{Proposition}

\theoremstyle{definition}

\newtheorem{question}[theorem]{Question}

\numberwithin{equation}{section}

\newcommand{\id}{{\rm id}}

\newcommand{\cW}{{\mathcal W}}

\newcommand{\cC}{{\mathcal C}}
\newcommand{\cU}{{\mathcal U}}

\newcommand{\cR}{{\mathcal R}}
\newcommand{\cP}{{\mathcal P}}

\newcommand{\cZ}{{\mathcal Z}}

\newcommand{\cV}{{\mathcal V}}

\newcommand{\sP}{{\mathscr P}}

\newcommand{\Cb}{{\mathbb C}}
\newcommand{\Zb}{{\mathbb Z}}

\newcommand{\Nb}{{\mathbb N}}

\newcommand{\eps}{\varepsilon}
\newcommand{\unit}{\mathbf{1}}

\DeclareMathOperator{\Act}{Act}

\DeclareMathOperator{\frmin}{FrMin}

\newcommand{\N}{\mathbb{N}}

\newcommand{\symd}{\triangle}

\newcommand{\acts}{\curvearrowright}

\newcommand{\ubar}[1]{\text{\b{$#1$}}}
\newcommand{\lBD}{\ubar{D}}
\newcommand{\uBD}{\bar{D}}
\newcommand{\suchthat}{:}
\newcommand{\from}{\colon}
\newcommand{\mathand}{\mbox{ and }}
\newcommand{\partialto}{\rightharpoonup}
\newcommand{\aug}{\mathrm{aug}}
\newcommand{\cG}{\mathcal{G}}
\newcommand{\pmp}{p{$.$}m{$.$}p{$.$}}
\DeclareMathOperator{\graph}{graph}
\DeclareMathOperator{\domain}{dom}

\allowdisplaybreaks

\begin{document}

\author[Conley]{Clinton T.~Conley}
\author[Jackson]{Steve C.~Jackson}
\author[Kerr]{David Kerr}
\author[Marks]{Andrew S.~Marks}
\author[Seward]{Brandon Seward}
\author[Tucker-Drob]{Robin D.~Tucker-Drob}

\address{\hskip-\parindent
Clinton T.~Conley, Department of Mathematical Sciences,
Carnegie Mellon University, Pittsburgh, PA 15213, U.S.A.}
\email{clintonc@andrew.cmu.edu}

\address{\hskip-\parindent
Steve Jackson, Department of Mathematics,
University of North Texas,
Denton, TX 76203-5017, U.S.A.}
\email{jackson@unt.edu}

\address{\hskip-\parindent
David Kerr, Department of Mathematics, Texas A{\&}M University,
College Station, TX 77843-3368, U.S.A.}
\email{kerr@math.tamu.edu}

\address{\hskip-\parindent
Andrew Marks, UCLA Department of Mathematics,
Los Angeles, CA 90095-1555, U.S.A.}
\email{marks@math.ucla.edu}

\address{\hskip-\parindent
Brandon Seward, Courant Institute of Mathematical Sciences,
New York, NY 10012, U.S.A.}
\email{bseward@cims.nyu.edu}

\address{\hskip-\parindent
Robin Tucker-Drob, Department of Mathematics, Texas A{\&}M University,
College Station, TX 77843-3368, U.S.A.}
\email{rtuckerd@math.tamu.edu}

\title{F{\o}lner tilings for actions of amenable groups}

\date{November 9, 2017}

\begin{abstract}
We show that every probability-measure-preserving action of a countable amenable group $G$
can be tiled, modulo a null set, using finitely many finite subsets
of $G$ (``shapes'') with prescribed approximate
invariance so that the collection of tiling centers for each shape is Borel.
This is a dynamical version of the Downarowicz--Huczek--Zhang tiling theorem for countable amenable groups
and strengthens the Ornstein--Weiss Rokhlin lemma. As an application we prove that, for every
countably infinite amenable group $G$, the crossed product of a generic free minimal action
of $G$ on the Cantor set is $\cZ$-stable.
\end{abstract}

\maketitle

\section{Introduction}

A discrete group $G$ is said to be {\it amenable} if it admits a finitely additive
probability measure which is invariant under the action of $G$ on itself by left translation,
or equivalently if there exists a unital positive linear functional $\ell^\infty (G)\to\Cb$
which is invariant under the action of $G$ on $\ell^\infty (G)$ induced by left translation
(such a functional is called a {\it left invariant mean}).
This definition was introduced by von Neumann in connection with the Banach--Tarski paradox
and shown by Tarski to be equivalent to the absence of paradoxical decompositions of the group.
Amenability has come to be most usefully leveraged through its
combinatorial expression as the F{\o}lner property,
which asks that for every finite set $K\subseteq G$ and
$\delta > 0$ there exists a nonempty finite set $F\subseteq G$ which is {\it $(K,\delta )$-invariant}
in the sense that $|KF\Delta F| < \delta |F|$.

The concept of amenability appears as a common thread throughout much of ergodic theory as well as the
related subject of operator algebras, where it is known via a number of avatars like
injectivity, hyperfiniteness, and nuclearity.
It forms the cornerstone of the theory of orbit equivalence,
and also underpins both Kolmogorov--Sinai entropy and the classical ergodic theorems,
whether explicitly in their most general formulations or implicitly
in the original setting of single transformations (see Chapters~4 and 9 of \cite{KerLi17}).
A key tool in applying amenability to dynamics is the Rokhlin lemma of Ornstein and Weiss,
which in one of its simpler forms says that for every free probability-measure-preserving action
$G\curvearrowright (X,\mu )$ of a countably infinite amenable group
and every finite set $K\subseteq G$ and $\delta > 0$ there exist $(K,\delta )$-invariant finite sets
$T_1 , \dots , T_n \subseteq G$ and measurable sets $A_1 , \dots , A_n \subseteq X$
such that the sets $sA_i$ for $i=1,\dots , n$ and $s\in T_i$ are pairwise disjoint
and have union of measure at least $1-\delta$ \cite{OrnWei87}.

The proportionality in terms of which approximate invariance is expressed in the F{\o}lner condition
makes it clear that amenability is a measure-theoretic property, and it is not surprising
that the most influential and definitive applications of these ideas in dynamics
(e.g., the Connes--Feldman--Weiss theorem) occur in the presence
of an invariant or quasi-invariant measure. Nevertheless, amenability also has significant
ramifications for topological dynamics, for instance in guaranteeing the existence
of invariant probability measures when the space is compact and in providing the basis
for the theory of topological entropy. In the realm of operator algebras, similar comments can be made
concerning the relative significance of amenability for von Neumann algebras (measure)
and C$^*$-algebras (topology).

While the subjects of von Neumann algebras and C$^*$-algebras have long enjoyed a symbiotic relationship
sustained in large part through the lens of analogy, and a similar relationship
has historically bound together ergodic theory and topological dynamics,
the last few years have witnessed the emergence of a new and structurally more direct
kind of rapport between topology and measure in these domains, beginning on the operator algebra side
with the groundbreaking work of Matui and Sato on strict comparison, $\cZ$-stability,
and decomposition rank \cite{MatSat12,MatSat14}.
On the side of groups and dynamics, Downarowicz, Huczek, and Zhang recently showed
that if $G$ is a countable amenable group then for every finite set $K\subseteq G$ and $\delta > 0$
one can partition (or ``tile'') $G$ by left translates of finitely many $(K,\delta )$-invariant
finite sets \cite{DowHucZha16}.
The consequences that they derive from this tileability are topological and include the
existence, for every such $G$, of a free minimal action with zero entropy.
One of the aims of the present paper is to provide some insight into how these advances
in operator algebras and dynamics, while seemingly unrelated at first glance,
actually fit together as part of a common circle of ideas that we expect, among other things,
to lead to further progress in the structure and classification theory of crossed product C$^*$-algebras.

Our main theorem is a version of the
Downarowicz--Huczek--Zhang tiling result for free p.m.p.\ (probability-measure-preserving) actions
of countable amenable groups which strengthens the Ornstein--Weiss Rokhlin lemma
in the form recalled above by shrinking the leftover piece down to a null set
(Theorem~\ref{T-main}). As in the case of groups, one does not expect
the utility of this dynamical tileability to be found in the measure setting,
where the Ornstein--Weiss machinery generally suffices, but rather in the derivation of
topological consequences.
Indeed we will apply our tiling result to show that,
for every countably infinite amenable group $G$,
the crossed product $C(X)\rtimes G$ of a generic free minimal action $G\curvearrowright X$ on the Cantor set
possesses the regularity property of $\cZ$-stability (Theorem~\ref{T-comeager}).
The strategy is to first prove that such an action admits clopen tower decompositions
with arbitrarily good F{\o}lner shapes (Theorem~\ref{T-dense G delta}),
and then to demonstrate that the existence of such
tower decompositions implies that the crossed product is $\cZ$-stable (Theorem~\ref{T-Z-stability}).
The significance of $\cZ$-stability within the classification program for simple separable nuclear C$^*$-algebras
is explained at the beginning of Section~\ref{S-Z-stability}.

It is a curious irony in the theory of amenability that the Hall--Rado matching theorem
can be used not only to show that the failure of the F{\o}lner property for a discrete group
implies the formally stronger Tarski characterization of nonamenability in terms of
the existence of paradoxical decompositions \cite{CecGriHar99}
but also to show, in the opposite direction, that the F{\o}lner property itself implies the formally
stronger Downarowicz--Huczek--Zhang characterization of amenability which guarantees the existence of
tilings of the group by translates of finitely many F{\o}lner sets \cite{DowHucZha16}.
This Janus-like scenario will be reprised here in the dynamical context through the use
of a measurable matching argument of Lyons and Nazarov that was originally developed
to prove that for every simple bipartite nonamenable Cayley graph of a discrete group $G$
there is a factor of a Bernoulli action of $G$ which is an a.e.\ perfect matching of the graph \cite{LyoNaz11}.
Accordingly the basic scheme for proving Theorem~\ref{T-main} will be the same as that of
Downarowicz, Huczek, and Zhang and divides into two parts:
\begin{enumerate}
\item using an Ornstein--Weiss-type argument
to show that a subset of the space of lower Banach density close to one can be tiled
by dynamical translates of F{\o}lner sets, and

\item using a Lyons--Nazarov-type measurable matching to distribute almost all
remaining points to existing tiles with only a small proportional increase in the size of
the F{\o}lner sets, so that the approximate invariance is preserved.
\end{enumerate}

We begin in Section~\ref{S-matchings} with the
measurable matching result (Lemma~\ref{lem:expansive}), which is a variation on the Lyons--Nazarov theorem from \cite{LyoNaz11}
and is established along similar lines.
In Section~\ref{S-tilings} we establish the appropriate variant of the Ornstein--Weiss
Rokhlin lemma (Lemma~\ref{lem:bdense}) and put everything together in
Theorem~\ref{T-main}.
Section~\ref{S-Cantor} contains the genericity result for free minimal actions on the
Cantor set, while Section~\ref{S-Z-stability} is devoted to the material on $\cZ$-stability.
\medskip

\noindent{\it Acknowledgements.}
C.C. was partially supported by NSF grant DMS-1500906. D.K. was partially supported by NSF grant DMS-1500593.
Part of this work was carried out while
he was visiting the Erwin Schr{\"o}dinger Institute (January--February 2016)
and the Mittag--Leffler Institute (February--March 2016). A.M. was
partially supported by NSF grant DMS-1500974. B.S. was partially supported by ERC grant 306494.
R.T.D. was partially supported by NSF grant DMS-1600904.
Part of this work was carried out during the AIM SQuaRE: Measurable Graph Theory.

\section{Measurable matchings}\label{S-matchings}

Given sets $X$ and $Y$ and a subset $\cR \subseteq X \times Y$, with each $x \in X$ we associate its \emph{vertical section} $\cR_x = \{y \in Y \suchthat (x,y)\in \cR\}$ and with each $y \in Y$ we associate its \emph{horizontal section} $\cR^y = \{x \in X \suchthat (x,y) \in \cR\}$.  Analogously, for $A \subseteq X$ we put $\cR_A = \bigcup_{x \in A} \cR_x = \{y \in Y \suchthat \exists x \in A\ (x,y) \in \cR\}$.  We say that $\cR$ is \emph{locally finite} if for all $x \in X$ and $y \in Y$ the sets $\cR_x$ and $\cR^y$ are finite.

If now $X$ and $Y$ are standard Borel spaces equipped with respective Borel measures $\mu$ and $\nu$, we say that $\cR \subseteq X \times Y$ is \emph{$(\mu,\nu)$-preserving} if whenever $f \from A \to B$ is a Borel bijection between subsets $A \subseteq X$ and $B \subseteq Y$ with $\graph(f) \subseteq \cR$ we have $\mu(A) = \nu(B)$.  We say that $\cR$ is \emph{expansive} if there is some $c > 1$ such that for all Borel $A \subseteq X$ we have $\nu(\cR_A) \geq c\mu(A)$.

We use the notation $f \from X \partialto Y$ to denote a partial function from $X$ to $Y$.  We say that such a partial function $f$ is \emph{compatible} with $\cR \subseteq X \times Y$ if $\graph(f) \subseteq \cR$.

\begin{proposition}[ess.~Lyons--Nazarov {\cite[Theorem 1.1]{LyoNaz11}}] \label{prop:match}
  Suppose that $X$ and $Y$ are standard Borel spaces, that $\mu$ is a Borel probability measure on $X$, and that $\nu$ is a Borel measure on $Y$.  Suppose that $\cR \subseteq X \times Y$ is Borel, locally finite, $(\mu,\nu)$-preserving, and expansive.  Then there is a $\mu$-conull $X' \subseteq X$ and a Borel injection $f \from X' \to Y$ compatible with $\cR$.
\end{proposition}

\begin{proof}
  Fix a constant of expansivity $c>1$ for $\cR$.

  We construct a sequence $(f_n)_{n \in \N}$ of Borel partial injections from $X$ to $Y$ which are compatible with $\cR$.  Moreover, we will guarantee that the set $X' = \{x \in X \suchthat \exists m \in \N \ \forall n \geq m\ x \in \domain(f_n) \mathand f_n(x) = f_m(x)\}$ is $\mu$-conull, establishing that the limiting function satisfies the conclusion of the lemma.

  Given a Borel partial injection $g \from X \partialto Y$ we say that a sequence $(x_0, y_0, \ldots, x_n, y_n) \in X \times Y \times \cdots \times X \times Y$ is a \emph{$g$-augmenting path} if
  \begin{itemize}
    \item
    $x_0 \in X$ is not in the domain of $g$,

    \item
    for all distinct $i,j < n$, $y_i \neq y_j$,

    \item
    for all $i < n$, $(x_i, y_i) \in \cR$,

    \item
    for all $i < n$, $y_i = g(x_{i+1})$,

    \item
    $y_n \in Y$ is not in the image of $g$.
  \end{itemize}
  We call $n$ the \emph{length} of such a $g$-augmenting path and $x_0$ the \emph{origin} of the path.  Note that the sequence $(x_0, y_0, y_1, \ldots, y_n)$ in fact determines the entire $g$-augmenting path, and moreover that $x_i \neq x_j$ for distinct $i,j < n$.
	
In order to proceed we require a few lemmas.

  \begin{lemma}\label{lemma:augdomain}
    Suppose that $n \in \N$ and $g \from X \partialto Y$ is a Borel partial injection compatible with $\cR$ admitting no augmenting paths of length less than $n$.  Then $\mu(X \setminus \domain(g)) \leq c^{-n}$.
  \end{lemma}
  \begin{proof}
    Put $A_0 = X \setminus \domain(g)$.  Define recursively for $i < n$ sets $B_i = \cR_{A_i}$ and $A_{i+1} = A_i \cup g^{-1}(B_i)$.  Note that the assumption that there are no augmenting paths of length less than $n$ implies that each $B_i$ is contained in the image of $g$.  Expansivity of $\cR$ yields $\nu(B_i) \geq c\mu(A_i)$ and $(\mu,\nu)$-preservation of $\cR$ then implies that $\mu(A_{i+1}) \geq \nu(B_i) \geq c\mu(A_i)$.  Consequently, $1 \geq \mu(A_n) \geq c^n \mu(A_0)$, and hence $\mu(A_0) \leq c^{-n}$.
  \end{proof}

  We say that a graph $\cG$ on a standard Borel space $X$ has a \emph{Borel $\N$-coloring} if there is a Borel function $c \from X \to \N$ such that if $x$ and $y$ are $\cG$-adjacent then $c(x) \neq c(y)$.

  \begin{lemma}[Kechris--Solecki--Todorcevic {\cite[Proposition 4.5]{KecSolTod99}}]\label{lemma:locfinchrom}
    Every locally finite Borel graph on a standard Borel space has a Borel $\N$-coloring.
  \end{lemma}

  \begin{proof}
    Fix a countable algebra $\{B_n \suchthat n \in \N\}$ of Borel sets which separates points (for example, the algebra generated by the basic open sets of a compatible Polish topology), and color each vertex $x$ by the least $n \in \N$ such that $B_n$ contains $x$ and none of its neighbors.
  \end{proof}

  Analogously, for $k \in \N$, we say that a graph on a standard Borel $X$ has a \emph{Borel $k$-coloring} if there is a Borel function $c \from X \to \{1,\ldots, k\}$ giving adjacent points distinct colors.

  \begin{lemma}[Kechris--Solecki--Todorcevic {\cite[Proposition 4.6]{KecSolTod99}}]\label{lemma:bdddegchrom}
    If a Borel graph on a standard Borel $X$ has degree bounded by $d \in \N$, then it has a Borel $(d+1)$-coloring.
  \end{lemma}

  \begin{proof}
    By Lemma \ref{lemma:locfinchrom}, the graph has a Borel $\N$-coloring $c \from X \to \N$.  We recursively build sets $A_n$ for $n \in \N$ by $A_0 = \{x \in X \suchthat c(x) = 0\}$ and $A_{n+1} = A_n \cup \{x \in X \suchthat c(x) = n+1 \mbox{ and no neighbor of $x$ is in $A_n$}\}$.  Then $A = \bigcup_n A_n$ is a Borel set which is $\cG$-independent, and moreover is maximal with this property.  So the restriction of $\cG$ to $X \setminus A$ has degree less than $d$, and the result follows by induction.
  \end{proof}

  \begin{lemma}[ess.~Elek--Lippner {\cite[Proposition 1.1]{EleLip10}}]\label{lemma:nosmallaug}
    Suppose that $g \from X \partialto Y$ is a Borel partial injection compatible with $\cR$, and let $n \geq 1$.  Then there is a Borel partial injection $g' \from X \partialto Y$ compatible with $\cR$ such that
    \begin{itemize}
      \item
      $\domain(g') \supseteq \domain(g)$,

      \item
      $g'$ admits no augmenting paths of length less than $n$,

      \item
      $\mu(\{x \in X \suchthat g'(x) \neq g(x)\} \leq n\mu(X \setminus \domain(g))$.
    \end{itemize}
  \end{lemma}
  \begin{proof}
    Consider the set $Z$ of injective sequences $(x_0,y_0,x_1,y_1,\ldots,x_m,y_m)$, where $m<n$, such that for all $i \leq m$ we have $(x_i, y_i) \in \cR$ and for all $i<m$ we have $(x_{i+1},y_i) \in \cR$.  Equip $Z$ with the standard Borel structure it inherits as a Borel subset of $(X \times Y)^{\leq n}$.  Consider also the locally finite Borel graph $\cG$ on $Z$ rendering adjacent two distinct sequences in $Z$ if they share any entries.  By Lemma \ref{lemma:locfinchrom} there is a partition $Z = \bigsqcup_{k \in \N} Z_k$ of $Z$ into Borel sets such that for all $k$, no two elements of $Z_k$ are $\cG$-adjacent.  In other words, we partition potential augmenting paths into countably many colors, where no two paths of the same color intersect.  Thus we may flip paths of the same color simultaneously without risk of causing conflicts.  Towards that end, fix a bookkeeping function $s \from \N \to \N$ such that $s^{-1}(k)$ is infinite for all $k \in \N$ in order to consider each color class infinitely often.

    Given a $g$-augmenting path $z = (x_0, y_0, \ldots, x_m, y_m)$, define the \emph{flip} along $z$ to be the Borel partial function $g_z \from X \partialto Y$ given by
    $$
    g_z(x) = \begin{cases}
      y_i & \mbox{ if } \exists i \leq m\ x = x_i,\\
      g(x) & \mbox{ otherwise}.
    \end{cases}
    $$
    The fact that $z$ is $g$-augmenting ensures that $g_z$ is injective.  More generally, for any Borel $\cG$-independent set $Z^\aug \subseteq Z$ of $g$-augmenting paths, we may simultaneously flip $g$ along all paths in $Z^\aug$ to obtain another Borel partial injection $(g)_{Z^\aug}$.

    We iterate this construction.  Put $g_0 = g$.  Recursively assuming that $g_k \from X \partialto Y$ has been defined, let $Z^\aug_k$ be the set of $g_k$-augmenting paths in $Z_{s(k)}$, and let $g_{k+1} = (g_k)_{Z^\aug_k}$ be the result of flipping $g_k$ along all paths in $Z^\aug_k$.  As each $x \in X$ is contained in only finitely many elements of $Z$, and since each path in $Z$ can be flipped at most once (after the first flip its origin is always in the domain of the subsequent partial injections), it follows that the sequence $(g_k(x))_{k \in \N}$ is eventually constant.  Defining $g'(x)$ to be the limiting value, it is routine to check that there are no $g'$-augmenting paths of length less than $n$.

    Finally, to verify the third item of the lemma, put $A = \{x \in X \suchthat g'(x) \neq g(x)\}$.  With each $x \in A$ associate the origin of the first augmenting path along which it was flipped.  This is an at most $n$-to-$1$ Borel function from $A$ to $X \setminus \domain(g)$, and since $\cR$ is $(\mu,\nu)$-preserving the bound follows.
  \end{proof}

  We are now in position to follow the strategy outlined at the beginning of the proof.  Let $f_0 \from X \partialto Y$ be the empty function.  Recursively assuming the Borel partial injection $f_n \from X \partialto Y$ has been defined to have no augmenting paths of length less than $n$, let $f_{n+1}$ be the Borel partial injection $(f_n)'$ granted by applying Lemma \ref{lemma:nosmallaug} to $f_n$.  Thus $f_{n+1}$ has no augmenting paths of length less than $n+1$ and the recursive construction continues.

  Lemma \ref{lemma:augdomain} ensures that $\mu(X \setminus \domain(f_n)) \leq c^{-n}$, and thus the third item of Lemma \ref{lemma:nosmallaug} ensures that $\mu(\{x \in X \suchthat f_{n+1}(x) \neq f_n(x)\}) \leq (n+1)c^{-n}$.  As the sequence $(n+1)c^{-n}$ is summable, the Borel--Cantelli lemma implies that $X' = \{x \in X \suchthat \exists m \in \N \ \forall n \geq m\ x \in \domain(f_n) \mathand f_n(x) = f_m(x)\}$ is $\mu$-conull.  Finally, $f = \lim_{n \to \infty} f_n \restriction X'$ is as desired.
\end{proof}

\begin{lemma} \label{lem:expansive}
Suppose $X$ and $Y$ are standard Borel spaces, that $\mu$ is a Borel measure on $X$, and that $\nu$ is a Borel measure on $Y$. Suppose $\cR \subseteq X\times Y$ is Borel, locally finite, $(\mu ,\nu )$-preserving graph. Assume that there exist numbers $a,b >0$ such that $|\cR _x|\geq a$ for $\mu$-a.e.\ $x\in X$ and $|\cR ^y|\leq b$ for $\nu$-a.e.\ $y\in Y$. Then $\nu (\cR _A )\geq \frac{a}{b}\mu (A)$ for all Borel subsets $A\subseteq X$.
\end{lemma}

\begin{proof}
Since $\cR$ is $(\mu ,\nu )$-preserving we have $\int _{A}|\cR _x | \, d\mu = \int _{\cR _A} |\cR ^y \cap A| \, d\nu$. Hence
\[
a\mu (A) =\int _{A} a \, d\mu \leq \int _{A}|\cR _x | \, d\mu = \int _{\cR _A} |\cR ^y \cap A| \, d\nu \leq \int _{\cR _A}b \, d\nu = b \nu (\cR _A)  .\qedhere
\]
\end{proof}

\section{F{\o}lner tilings}\label{S-tilings}

Fix a countable group $G$. For finite sets $K, F \subseteq G$ and $\delta > 0$, we say that $F$ is \emph{$(K, \delta)$-invariant} if $|K F \symd F| < \delta |F|$. Note this condition implies $|K F| < (1 + \delta) |F|$. Recall that $G$ is \emph{amenable} if for every finite $K \subseteq G$ and $\delta > 0$ there exists a $(K, \delta)$-invariant set $F$. A \emph{F{\o}lner sequence} is a sequence of finite sets $F_n \subseteq G$ with the property that for every finite $K \subseteq G$ and $\delta > 0$ the set $F_n$ is $(K, \delta)$-invariant for all but finitely many $n$. Below, we always assume that $G$ is amenable.

Fix a free action $G \acts X$. For $A \subseteq X$ we define the lower and upper \emph{Banach densities} of $A$ to be
$$\lBD(A) = \sup_{\substack{F \subseteq G\\F \text{ finite}}} \inf_{x \in X} \frac{|A \cap F x|}{|F|} \qquad \text{and} \qquad \uBD(A) = \inf_{\substack{F \subseteq G\\F \text{ finite}}} \sup_{x \in X} \frac{|A \cap F x|}{|F|}.$$
Equivalently \cite[Lemma 2.9]{DowHucZha16}, if $(F_n)_{n \in \N}$ is a F{\o}lner sequence then
$$\lBD(A) = \lim_{n \rightarrow \infty} \inf_{x \in X} \frac{|A \cap F_n x|}{|F_n|} \qquad \text{and} \qquad \uBD(A) = \lim_{n \rightarrow \infty} \sup_{x \in X} \frac{|A \cap F_n x|}{|F_n|}.$$

We now define an analogue of `$(K, \delta)$-invariant' for infinite subsets of $X$. A set $A \subseteq X$ (possibly infinite) is \emph{$(K, \delta)^*$-invariant} if there is a finite set $F \subseteq G$ such that $|(K A \symd A) \cap F x| < \delta |A \cap F x|$ for all $x \in X$. Equivalently, $A$ is $(K, \delta)^*$-invariant if and only if for every F{\o}lner sequence $(F_n)_{n \in \N}$ we have $\lim_n \sup_x |(K A \symd A) \cap F_n x| / |A \cap F_n x| < \delta$.

A collection $\{F_i : i\in I \}$ of finite subsets of $X$ is called \emph{$\epsilon$-disjoint} if for each $i$ there is an $F_i' \subseteq F_i$ such that $|F_i'| > (1 - \epsilon)|F_i|$ and such that the sets $\{F_i' : i\in I\}$ are pairwise disjoint.

\begin{lemma} \label{lem:starinv}
Let $K, W \subseteq G$ be finite, let $\epsilon, \delta > 0$, let $C \subseteq X$, and for $c \in C$ let $F_c \subseteq W$ be $(K, \delta (1 - \epsilon))$-invariant. If the collection $\{F_c c : c \in C\}$ is $\epsilon$-disjoint and $\bigcup_{c \in C} F_c c$ has positive lower Banach density, then $\bigcup_{c \in C} F_c c$ is $(K, \delta)^*$-invariant.
\end{lemma}

\begin{proof}
Set $A = \bigcup_{c \in C} F_c c$ and set $T = W W^{-1} (\{1_G\} \cup K)^{-1}$. Since $W$ is finite and each $F_c \subseteq W$, there is $0 < \delta_0 < \delta$ such that each $F_c$ is $(K, \delta_0 (1 - \epsilon))$-invariant. Fix a finite set $U \subseteq G$ which is $(T, \frac{\lBD(A)}{2 |T|} (\delta - \delta_0))$-invariant and satisfies $\inf_{x \in X} |A \cap U x| > \frac{\lBD(A)}{2} |U|$. Now fix $x \in X$. Let $B$ be the set of $b \in U x$ such that $T b \not\subseteq U x$. Note that $B \subseteq T^{-1} (T U x \symd U x)$ and thus
$$|B| \leq |T| \cdot \frac{|T U \symd U|}{|U|} \cdot \frac{|U|}{|A \cap U x|} \cdot |A \cap U x| < (\delta - \delta_0) |A \cap U x|.$$
Set $C' = \{c \in C : F_c c \subseteq U x\}$. Note that the $\epsilon$-disjoint assumption gives $(1 - \epsilon) \sum_{c \in C'} |F_c| \leq |A \cap U x|$. Also, our definitions of $C'$, $T$, and $B$ imply that if $c \in C \setminus C'$ and $(\{1_G\} \cup K) F_c c \cap U x \neq \emptyset$ then $((\{1_G\} \cup K) F_c c) \cap U x \subseteq B$. Therefore $(K A \symd A) \cap U x \subseteq B \cup \bigcup_{c \in C'} (K F_c c \symd F_c c)$. Combining this with the fact that each set $F_c$ is $(K, \delta_0 (1 - \epsilon))$-invariant, we obtain
\begin{align*}
|(K A \symd A) \cap U x| & \leq |B| + \sum_{c \in C'} |K F_c c \symd F_c c|\\
 & < (\delta - \delta_0) |A \cap U x| + \sum_{c \in C'} \delta_0 (1 - \epsilon) |F_c|\\
 & \leq (\delta - \delta_0) |A \cap U x| + \delta_0 |A \cap U x|\\
 & = \delta |A \cap U x|.
\end{align*}
Since $x$ was arbitrary, we conclude that $A$ is $(K, \delta)^*$-invariant.
\end{proof}

\begin{lemma} \label{lem:pack}
Let $T \subseteq G$ be finite and let $\epsilon, \delta > 0$ with $\epsilon (1 + \delta) < 1$. Suppose that $A \subseteq X$ is $(T^{-1}, \delta)^*$-invariant. If $B \supseteq A$ and $|B \cap T x| \geq \epsilon |T|$ for all $x \in X$, then
$$\lBD(B) \geq (1 - \epsilon (1 + \delta)) \cdot \lBD(A) + \epsilon.$$
\end{lemma}

\begin{proof}
This is implicitly demonstrated in \cite[Proof of Lemma 4.1]{DowHucZha16}. As a convenience to the reader, we include a proof here. Fix $\theta > 0$. Since $A$ is $(T^{-1}, \delta)^*$-invariant, we can pick a finite set $U \subseteq G$ which is $(T, \theta)$-invariant and satisfies
$$\inf_{x \in X} \frac{|A \cap U x|}{|U|} > \lBD(A) - \theta \qquad \text{and} \qquad \sup_{x \in X} \frac{|T^{-1} A \cap U x|}{|A \cap U x|} < 1 + \delta.$$
Fix $x \in X$, set $\alpha = \frac{|A \cap U x|}{|U|} > \lBD(A) - \theta$, and set $U' = \{u \in U : A \cap T u x = \emptyset\}$. Notice that
$$\frac{|U'|}{|U|} = \frac{|U| - |T^{-1} A \cap U x|}{|U|} = 1 - \frac{|T^{-1} A \cap U x|}{|A \cap U x|} \cdot \frac{|A \cap U x|}{|U|} > 1 - (1 + \delta) \alpha.$$
Since $A \cap T U' x = \emptyset$ and $|B \cap T y| \geq \epsilon |T|$ for all $y \in X$, it follows that $|(B \setminus A) \cap T u x| \geq \epsilon |T|$ for all $u \in U'$. Thus there are $\epsilon |T| |U'|$ many pairs $(t, u) \in T \times U'$ with $t u x \in B \setminus A$. It follows there is $t^* \in T$ with $|(B \setminus A) \cap t^* U' x| \geq \epsilon \cdot |U'|$. Therefore
\begin{align*}
\frac{|B \cap T U x|}{|T U|} & \geq \left( \frac{|A \cap U x|}{|U|} + \frac{|(B \setminus A) \cap t^* U' x|}{|U'|} \cdot \frac{|U'|}{|U|} \right) \cdot \frac{|U|}{|T U|}\\
 & > \Big( \alpha + \epsilon (1 - (1 + \delta) \alpha) \Big) \cdot (1 + \theta)^{-1}\\
 & = \Big((1 - \epsilon (1 + \delta)) \alpha + \epsilon \Big) \cdot (1 + \theta)^{-1}\\
 & > \Big((1 - \epsilon (1 + \delta)) (\lBD(A) - \theta) + \epsilon \Big) \cdot (1 + \theta)^{-1}.
\end{align*}
Letting $\theta$ tend to $0$ completes the proof.
\end{proof}

\begin{lemma} \label{lem:bepack}
Let $X$ be a standard Borel space and let $G \acts X$ be a free Borel action. Let $Y \subseteq X$ be Borel, let $T \subseteq G$ be finite, and let $\epsilon \in (0, 1 / 2)$. Then there is a Borel set $C \subseteq X$ and a Borel function $c \in C \mapsto T_c \subseteq T$ such that $|T_c| > (1 - \epsilon) |T|$, the sets $\{T_c c : c \in C\}$ are pairwise disjoint and disjoint with $Y$, $Y \cup \bigcup_{c \in C} T_c c = Y \cup T C$, and $|(Y \cup T C) \cap T x| \geq \epsilon |T|$ for all $x \in X$.
\end{lemma}

\begin{proof}
Using Lemma \ref{lemma:bdddegchrom}, fix a Borel partition $\cP = \{P_1, \ldots, P_m\}$ of $X$ such that $T x \cap T x' = \emptyset$ for all $x \neq x' \in P_i$ and all $1 \leq i \leq m$. We will pick Borel sets $C_i \subseteq P_i$ and set $C = \bigcup_{1 \leq i \leq m} C_i$. Set $Y_0 = Y$. Let $1 \leq i \leq m$ and inductively assume that $Y_{i-1}$ has been defined. Define $C_i = \{c \in P_i : |Y_{i-1} \cap T c| < \epsilon |T|\}$, define $Y_i = Y_{i-1} \cup T C_i$, and for $c \in C_i$ set $T_c = \{t \in T : t c \not\in Y_{i-1}\}$. It is easily seen that $C = \bigcup_{1 \leq i \leq m} C_i$ has the desired properties.
\end{proof}

The following lemma is mainly due to Ornstein--Weiss \cite{OrnWei87}, who proved
it with an invariant probability measure taking the place of Banach density.
Ornstein and Weiss also established a purely group-theoretic counterpart
of this result which was later adapted to the Banach density setting by
Downarowicz--Huczek--Zhang in \cite{DowHucZha16}
and will be heavily used in Section~\ref{S-Z-stability},
where it is recorded as Theorem~\ref{T-qt}.
The only difference between this lemma and prior versions is that
we simultaneously work in the Borel setting and use Banach density.

\begin{lemma}
\cite[II.\S 2. Theorem 5]{OrnWei87} \cite[Lemma 4.1]{DowHucZha16} \label{lem:bdense}
Let $X$ be a standard Borel space and let $G \acts X$ be a free Borel action. Let $K \subseteq G$ be finite, let $\epsilon \in (0, 1 / 2)$, and let $n$ satisfy $(1 - \epsilon)^n < \epsilon$. Then there exist $(K, \epsilon)$-invariant sets $F_1, \ldots, F_n$, a Borel set $C \subseteq X$, and a Borel function $c \in C \mapsto F_c \subseteq G$ such that:
\begin{enumerate}
\item[\rm (i)] for every $c \in C$ there is $1 \leq i \leq n$ with $F_c \subseteq F_i$ and $|F_c| > (1 - \epsilon) |F_i|$;
\item[\rm (ii)] the sets $F_c c$, $c \in C$, are pairwise disjoint; and
\item[\rm (iii)] $\lBD(\bigcup_{c \in C} F_c c) > 1 - \epsilon$.
\end{enumerate}
\end{lemma}

\begin{proof}
Fix $\delta > 0$ satisfying $(1 + \delta)^{-1}(1 - (1 + \delta) \epsilon)^n < \epsilon - 1 + (1 + \delta)^{-1}$. Fix a sequence of $(K, \epsilon)$-invariant sets $F_1, \ldots, F_n$ such that $F_i$ is $(F_j^{-1}, \delta (1 - \epsilon))$-invariant for all $1 \leq j < i \leq n$.

The set $C$ will be the disjoint union of sets $C_i$, $1 \leq i \leq n$. The construction will be such that $F_c \subseteq F_i$ and $|F_c| > (1 - \epsilon) |F_i|$ for $c \in C_i$. We will define $A_i = \bigcup_{i \leq k \leq n} \bigcup_{c \in C_k} F_c c$ and arrange that $A_{i+1} \cup F_i C_i = A_{i+1} \cup \bigcup_{c \in C_i} F_c c$ and
\begin{equation} \label{eqn:bdense}
\lBD(A_i) \geq (1 + \delta)^{-1} - (1 + \delta)^{-1} (1 - \epsilon (1 + \delta))^{n + 1 - i}.
\end{equation}
In particular, we will have $A_i = \bigcup_{i \leq k \leq n} F_k C_k$.

To begin, apply Lemma \ref{lem:bepack} with $Y = \emptyset$ and $T = F_n$ to get a Borel set $C_n$ and a Borel map $c \in C_n \mapsto F_c \subseteq F_n$ such that $|F_c| > (1 - \epsilon) |F_n|$, the sets $\{F_c c : c \in C_n\}$ are pairwise disjoint, $\bigcup_{c \in C_n} F_c c = F_n C_n$, and $|F_n C_n \cap F_n x| \geq \epsilon |F_n|$ for all $x \in X$. Applying Lemma \ref{lem:pack} with $A = \emptyset$ and $B = F_n C_n$ we find that the set $A_n = F_n C_n$ satisfies $\lBD(A_n) \geq \epsilon$.

Inductively assume that $C_n$ through $C_{i+1}$ have been defined and $A_n$ through $A_{i+1}$ are defined as above and satisfy (\ref{eqn:bdense}). Using $Y = A_{i+1}$ and $T = F_i$, apply Lemma \ref{lem:bepack} to get a Borel set $C_i$ and a Borel map $c \in C_i \mapsto F_c \subseteq F_i$ such that $|F_c| > (1 - \epsilon) |F_i|$, the sets $\{F_c c : c \in C_i\}$ are pairwise disjoint and disjoint with $A_{i+1}$, $A_{i+1} \cup \bigcup_{c \in C_i} F_c c = A_{i+1} \cup F_i C_i$, and $|(A_{i+1} \cup F_i C_i) \cap F_i x| \geq \epsilon |F_i|$ for all $x \in X$. The set $A_{i+1}$ is the union of an $\epsilon$-disjoint collection of $(F_i^{-1}, \delta (1 - \epsilon))$-invariant sets and has positive lower Banach density. So by Lemma \ref{lem:starinv} $A_{i+1}$ is $(F_i^{-1}, \delta)^*$-invariant. Applying Lemma \ref{lem:pack} with $A = A_{i+1}$, we find that $A_i = A_{i + 1} \cup F_i C_i$ satisfies
\begin{align*}
\lBD(A_i) & \geq (1 - \epsilon (1 + \delta)) \cdot \lBD(A_{i+1}) + \epsilon\\
 & \geq \frac{(1 - \epsilon (1 + \delta))}{(1 + \delta)} - (1 + \delta)^{-1} (1 - \epsilon (1 + \delta))^{n + 1 - i} + \frac{\epsilon (1 + \delta)}{1 + \delta}\\
 & = (1 + \delta)^{-1} - (1 + \delta)^{-1} (1 - \epsilon (1 + \delta))^{n + 1 - i}.
\end{align*}
This completes the inductive step and completes the definition of $C$. It is immediate from the construction that (i) and (ii) are satisfied. Clause (iii) also follows by noting that (\ref{eqn:bdense}) is greater than $1 - \epsilon$ when $i = 1$.
\end{proof}

We recall the following simple fact.

\begin{lemma} \cite[Lemma 2.3]{DowHucZha16} \label{lem:pinv}
If $F \subseteq G$ is $(K, \delta)$-invariant and $F'$ satisfies $|F' \symd F| < \epsilon |F|$ then $F'$ is $(K, \frac{(|K| + 1) \epsilon + \delta}{1 - \epsilon})$-invariant.
\end{lemma}

Now we present the main theorem.

\begin{theorem}\label{T-main}
Let $G$ be a countable amenable group, let $(X, \mu)$ be a standard probability space, and let $G \acts (X, \mu)$ be a free {\pmp} action. For every finite $K \subseteq G$ and every $\delta > 0$ there exist a $\mu$-conull $G$-invariant Borel set $X' \subseteq X$, a collection $\{C_i : 0 \leq i \leq m\}$ of Borel subsets of $X'$, and a collection $\{F_i : 0 \leq i \leq m\}$ of $(K, \delta)$-invariant sets such that $\{F_i c : 0 \leq i \leq m, \ c \in C_i\}$ partitions $X'$.
\end{theorem}

\begin{proof}
Fix $\epsilon \in (0, 1 / 2)$ satisfying $\frac{(|K| + 1) 6 \epsilon + \epsilon}{1 - 6 \epsilon} < \delta$. Apply Lemma \ref{lem:bdense} to get $(K, \epsilon)$-invariant sets $F_1' ,\ldots, F_n'$, a Borel set $C \subseteq X$, and a Borel function $c \in C \mapsto F_c \subseteq G$ satisfying
\begin{enumerate}
\item[\rm (i)] for every $c \in C$ there is $1 \leq i \leq n$ with $F_c \subseteq F_i'$ and $|F_c| > (1 - \epsilon) |F_i'|$;
\item[\rm (ii)] the sets $F_c c$, $c \in C$, are pairwise disjoint; and
\item[\rm (iii)] $\lBD(\bigcup_{c \in C} F_c c) > 1 - \epsilon$.
\end{enumerate}
Set $Y = X \setminus \bigcup_{c \in C} F_c c$. If $\mu(Y) = 0$ then we are done. So we assume $\mu(Y) > 0$ and we let $\nu$ denote the restriction of $\mu$ to $Y$. Fix a Borel map $c \in C \mapsto Z_c \subseteq F_c$ satisfying $4 \epsilon |F_c| < |Z_c| < 5 \epsilon |F_c|$ for all $c \in C$ (it's clear from the proof of Lemma \ref{lem:bdense} that we may choose the sets $F_i'$ so that $\epsilon |F_c| > \min_i \epsilon (1 - \epsilon) |F_i'| > 1$). Set $Z = \bigcup_{c \in C} Z_c c$ and let $\zeta$ denote the restriction of $\mu$ to $Z$ (note that $\mu(Z) > 0$).

Set $W = \bigcup_{i=1}^n F_i'$ and $W' = W W^{-1}$. Fix a finite set $U \subseteq G$ which is $(W', (1/2 - \epsilon) / |W'|)$-invariant and satisfies $\inf_{x \in X} |(X \setminus Y) \cap U x| > (1 - \epsilon) |U|$. Since every amenable group admits a F{\o}lner sequence consisting of symmetric sets, we may assume that $U = U^{-1}$ \cite[Corollary 5.3]{Nam64}. Define $\cR \subseteq Y \times Z$ by declaring $(y, z) \in \cR$ if and only if $y \in U z$ (equivalently $z \in U y$). Then $\cR$ is Borel, locally finite, and $(\nu, \zeta)$-preserving. We now check that $\cR$ is expansive. We automatically have $|\cR^z| = |Y \cap U z| < \epsilon |U|$ for all $z \in Z$. By Lemma \ref{lem:expansive} it suffices to show that $|\cR_y| = |Z \cap U y| \geq 2 \epsilon |U|$ for all $y \in Y$. Fix $y \in Y$. Let $B$ be the set of $b \in U y$ such that $W' b \not\subseteq U y$. Then $B \subseteq W' (W' U y \symd U y)$ and thus
$$\frac{|B|}{|U|} \leq |W'| \cdot \frac{|W' U \symd U|}{|U|} < 1 / 2 - \epsilon.$$
Let $A$ be the union of those sets $F_c c$, $c \in C$, which are contained in $U y$. Notice that $(X \setminus Y) \cap U y \subseteq B \cup A$. Therefore
$$\frac{1}{2} - \epsilon + \frac{|A|}{|U|} > \frac{|(B \cup A) \cap U y|}{|U|} \geq \frac{|(X \setminus Y) \cap U y|}{|U|} > 1 - \epsilon,$$
hence $|A| > |U| / 2$. By construction $|Z \cap A| > 4 \epsilon |A|$. So $|Z \cap U y| \geq |Z \cap A| >  2 \epsilon |U|$. We conclude that $\cR$ is expansive.

Apply Proposition \ref{prop:match} to obtain a $G$-invariant $\mu$-conull set $X' \subseteq X$ and a Borel injection $\rho : Y \cap X' \rightarrow Z$ with $\graph(\rho) \subseteq \cR$. Consider the sets $F_c \cup \{g \in U : g c \in Y \text{ and } \rho(g c) \in F_c c\}$ as $c \in C$ varies. These are subsets of $W \cup U$ and thus there are only finitely many such sets which we can enumerate as $F_1, \ldots, F_m$. We partition $C \cap X'$ into Borel sets $C_1, \ldots, C_m$ with $c \in C_i$ if and only if $c \in X'$ and $F_c \cup \{g \in U : g c \in Y \text{ and } \rho(g c) \in F_c c\} = F_i$. Since $\rho$ is defined on all of $Y \cap X'$, we see that the sets $\{F_i c : 1 \leq i \leq m, \ c \in C_i\}$ partition $X'$. Finally, for $c \in C_i \cap X'$, if we let $F_j'$ be such that $|F_c \symd F_j'| < \epsilon |F_j'|$, then
\begin{align*}
|F_i \symd F_j'| \leq |F_i \symd F_c| + |F_c \symd F_j'|
&\leq |\rho^{-1}(F_c c)| + \epsilon |F_j'| \\
&\leq |Z_c| + \epsilon |F_j'|
< 5 \epsilon |F_c| + \epsilon |F_j'| \leq 6 \epsilon |F_j'|.
\end{align*}
Using Lemma \ref{lem:pinv} and our choice of $\epsilon$, this implies that each set $F_i$ is $(K, \delta)$-invariant.
\end{proof}

\section{Clopen tower decompositions with F{\o}lner shapes}\label{S-Cantor}

Let $G\curvearrowright X$ be an action of a group on a compact space.
By a {\it clopen tower} we mean a pair $(B,S)$ where $B$ is a clopen subset of $X$
(the {\it base} of the tower) and $S$ is a finite subset of $G$ (the {\it shape} of the tower)
such that the sets $sB$ for $s\in S$ are pairwise disjoint.
By a {\it clopen tower decomposition} of $X$ we mean a finite collection $\{ (B_i , S_i ) \}_{i=1}^n$
of clopen towers such that the sets $S_1 B_1 , \dots , S_n B_n$ form a partition of $X$.
We also similarly speak of {\it measurable towers} and {\it measurable tower decompositions}
for an action $G\curvearrowright (X,\mu )$ on a measure space, with the bases now being measurable sets
instead of clopen sets. In this terminology, Theorem~\ref{T-main} says that if
$G\curvearrowright (X,\mu )$ is a free p.m.p.\ action of a countable amenable group on
a standard probability space then for every finite set $K\subseteq G$ and $\delta > 0$
there exists, modulo a null set, a measurable tower decomposition of $X$ with $(K,\delta )$-invariant shapes.

\begin{lemma}\label{L-clopen}
Let $G$ be a countably infinite amenable group
and $G\curvearrowright X$ a free minimal action on the Cantor set.
Then this action has a free minimal extension $G\curvearrowright Y$ on
the Cantor set such that for every finite set $F\subseteq G$ and $\delta > 0$ there is a
clopen tower decomposition of $Y$ with $(F,\delta )$-invariant shapes.
\end{lemma}

\begin{proof}
Let $F_1 \subseteq F_2 \subseteq\dots$ be an increasing sequence of finite subsets of $G$
whose union is equal to $G$. Fix a $G$-invariant Borel probability measure $\mu$ on $X$
(such a measure exists by amenability). The freeness of the action $G\curvearrowright X$
means that for each $n\in\Nb$ we can apply Theorem \ref{T-main} to produce, modulo a null set,
a measurable tower decomposition $\cU_n$ for the p.m.p.\ action $G\curvearrowright (X,\mu )$
such that each shape is $(F_n ,1/n)$-invariant.
Let $A$ be the unital $G$-invariant C$^*$-algebra of $L^\infty (X,\mu )$ generated by
$C(X)$ and the indicator functions of the levels of each of the tower decompositions $\cU_n$.
Since there are countably many such indicator functions and the group $G$ is countable,
the C$^*$-algebra $A$ is separable.
Therefore by the Gelfand--Naimark theorem we have $A = C(Z)$ for some zero-dimensional metrizable space
$Z$ and a $G$-factor map $\varphi : Z \to X$.
By a standard fact which can be established using Zorn's lemma,
there exists a nonempty closed $G$-invariant set $Y \subseteq Z$
such that the restriction action $G\curvearrowright Y$ is minimal. Note that $Y$ is
necessarily a Cantor set, since $G$ is infinite. Also, the action $G\curvearrowright Y$ is free,
since it is an extension of a free action.
Since the action on $X$ is minimal, the restriction $\varphi |_Y : Y \to X$ is surjective
and hence a $G$-factor map.
For each $n$ we get from $\cU_n$ a clopen tower decomposition $\cV_n$ of $Y$ with $(F_n ,1/n)$-invariant shapes,
and by intersecting the levels of the towers in $\cV_n$ with $Y$
we obtain a clopen tower decomposition of $Y$ with $(F_n ,1/n)$-invariant shapes,
showing that the extension $G\curvearrowright Y$ has the desired property.
\end{proof}

Let $X$ be the Cantor set and let $G$ be a countable infinite amenable group.
The set $\Act (G,X)$ is a Polish space under the topology which has as a basis the sets
\[
U_{\alpha ,\cP ,F} = \{ \beta\in\Act (G,X) : \alpha_s A = \beta_s A \text{ for all } A\in\cP \text{ and } s\in F \}
\]
where $\alpha\in\Act (G,X)$, $\cP$ is a clopen partition of $X$, and $F$ is a finite subset of $G$.
Write $\frmin (G,X)$ for the set of actions in $\Act (G,X)$ which are free and minimal.
Then $\frmin (G,X)$ is a $G_\delta$ set. To see this, fix an enumeration $s_1 , s_2 , s_3 , \dots$ of $G\setminus \{ e \}$ (where $e$ denotes the identity element of the group)
and for every $n\in\Nb$ and nonempty clopen set $A\subseteq X$
define the set $\cW_{n,A}$ of all $\alpha\in\Act (G,X)$ such that (i)
$\bigcup_{s\in F} \alpha_s A = X$ for some finite set $F\subseteq G$,
and (ii) there exists a clopen partition $\{ A_1 , \dots , A_k \}$ of $A$ such that
$\alpha_{s_n} A_i \cap A_i = \emptyset$ for all $i=1,\dots ,k$.
Then each $\cW_{n,A}$ is open, which means, with $A$ ranging over the countable collection of
nonempty clopen subsets of $X$, that the intersection
$\bigcap_{n\in\Nb} \bigcap_A \cW_{n,A}$, which is equal to $\frmin (G,X)$, is a $G_\delta$ set.
It follows that $\frmin (G,X)$ is a Polish space.

\begin{theorem}\label{T-dense G delta}
Let $G$ be a countably infinite amenable group. Let $\cC$ be the collection of actions
in $\frmin (G,X)$ with the property that for every finite set $F\subseteq G$ and $\delta > 0$ there is a
clopen tower decomposition of $X$ with $(F,\delta )$-invariant shapes. Then
$\cC$ is a dense $G_\delta$ subset of $\frmin (G,X)$.
\end{theorem}

\begin{proof}
That $\cC$ is a $G_\delta$ set is a simple exercise. Let $G\stackrel{\alpha}{\curvearrowright} X$
be any action in $\frmin (G,X)$.
By Lemma~\ref{L-clopen} this action has a free minimal extension $G\stackrel{\beta}{\curvearrowright} Y$
with the property in the theorem statement, where $Y$ is the Cantor set. Let $\cP$ be a clopen partition of $X$
and $F$ a nonempty finite subset of $G$. Write $A_1 , \dots , A_n$ for the members of the clopen partition
$\bigvee_{s\in F} s^{-1} \cP$.
Then for each $i=1,\dots ,n$ the set $A_i$ and its inverse image $\varphi^{-1} (A_i)$ under the extension map
$\varphi :Y\to X$ are Cantor sets, and so we can find a homeomorphism $\psi_i : A_i \to \varphi^{-1} (A_i )$.
Let $\psi : X\to Y$ be the homeomorphism which is equal to $\psi_i$ on $A_i$ for each $i$.
Then the action $G\stackrel{\gamma}{\curvearrowright} X$ defined by $\gamma_s = \psi^{-1} \circ\beta_s \circ\psi$
for $s\in G$ belongs to $\cC$ as well as to the basic open neighborhood $U_{\alpha ,\cP ,F}$ of $\alpha$,
establishing the density of $\cC$.
\end{proof}

\section{Applications to $\cZ$-stability}\label{S-Z-stability}

A C$^*$-algebra $A$ is said to be {\it $\cZ$-stable} if $A\otimes\cZ \cong A$ where
$\cZ$ is the Jiang--Su algebra \cite{JiaSu99}, with
the C$^*$-tensor product being unique in this case because $\cZ$ is nuclear.
$\cZ$-stability has become an important regularity property in the classification program
for simple separable nuclear C$^*$-algebras, which has recently witnessed some spectacular advances.
Thanks to recent work of Gong--Lin--Niu \cite{GonLinNiu15},
Elliott--Gong--Lin--Niu \cite{EllGonLinNiu15},
and Tikuisis--White--Winter \cite{TikWhiWin17}, it is now known that simple separable unital C$^*$-algebras
satisfying the universal coefficient theorem and having finite nuclear dimension
are classified by ordered $K$-theory paired with tracial states. Although $\cZ$-stability
does not appear in the hypotheses of this classification theorem, it does play an important
technical role in the proof. Moreover, it is a conjecture of Toms and Winter that
for simple separable infinite-dimensional unital nuclear C$^*$-algebras
the following properties are equivalent:
\begin{enumerate}
\item $\cZ$-stability,

\item finite nuclear dimension,

\item strict comparison.
\end{enumerate}
Implications between (i), (ii), and (iii) are known to hold in various degrees of generality.
In particular, the implication (ii)$\Rightarrow$(i) was established in \cite{Win12} while the converse
is known to hold when the extreme boundary of the convex set of tracial states is compact \cite{BosBroSatTikWhiWin16}.
It remains a problem to determine whether any of the crossed products of the actions in Theorem~\ref{T-comeager}
falls within the purview of these positive results on the Toms--Winter conjecture,
and in particular whether any of them has finite nuclear dimension (see Question~\ref{Q-nuclear dimension}).

By now there exist highly effectively methods for establishing finite nuclear dimension
for crossed products of free actions on compact metrizable spaces of
finite covering dimension \cite{Sza15,SzaWuZac14,GueWilYu16},
but their utility is structurally restricted to groups with finite asymptotic dimension
and hence excludes many amenable examples like the Grigorchuk group. One can show using the
technology from \cite{GueWilYu16} that, for a countably infinite amenable group with finite asymptotic
dimension, the crossed product of a generic free minimal action on the Cantor set
has finite nuclear dimension. Our intention here has been to remove the restriction of finite
asymptotic dimension by means of a different approach that establishes instead the
conjecturally equivalent property of $\cZ$-stability but for arbitrary countably infinite amenable groups.

To verify $\cZ$-stability in the proof of Theorem~\ref{T-Z-stability}
we will use the following result of Hirshberg and Orovitz \cite{HirOro13}.
Recall that a linear map $\varphi :A\to B$ between C$^*$-algebras is said to be
{\it complete positive} if its tensor product
$\id\otimes \varphi : M_n \otimes A \to M_n \otimes B$ with the identity map
on the $n\times n$ matrix algebra $M_n$ is positive for every $n\in\Nb$.
It is of {\it order zero} if $\varphi (a)\varphi (b)  = 0$ for all $a,b\in A$ satisfying
$ab = 0$. One can show that $\varphi$ is an order-zero completely positive map
if and only if there is an embedding $B\subseteq D$ of $B$ into a larger C$^*$-algebra,
a $^*$-homomorphism $\pi : A\to D$, and a positive element $h\in D$ commuting
with the image of $\pi$ such that $\varphi (a) = h\pi (a)$ for all $a\in A$ \cite{WinZac09}.
Below $\precsim$ denotes the relation of Cuntz subequivalence, so that $a\precsim b$
for positive elements $a,b$ in a C$^*$-algebra $A$ means that
there is a sequence $( v_n )$ in $A$ such that $\lim_{n\to\infty} \| a - v_n b v_n^* \| = 0$.

\begin{theorem}\label{T-Z-stable}
Let $A$ be a simple separable unital nuclear C$^*$-algebra not isomorphic to $\Cb$.
Suppose that for every
$n\in\Nb$, finite set $\Omega\subseteq A$, $\eps > 0$, and nonzero positive element $a\in A$
there exists an order-zero complete positive contractive linear map
$\varphi : M_n \to A$ such that
\begin{enumerate}
\item $1-\varphi (1) \precsim a$,

\item $\| [b,\varphi (z)] \| < \eps$ for all $b\in\Omega$ and norm-one $z\in M_n$.
\end{enumerate}
Then $A$ is $\cZ$-stable.
\end{theorem}

The following is the Ornstein--Weiss quasitiling theorem \cite{OrnWei87} as formulated
in Theorem~3.36 of \cite{KerLi17}. For finite sets $A,F\subseteq G$ we write
\[
\partial_F A = \{ s\in A : Fs \cap A\neq\emptyset\text{ and } Fs \cap (G\setminus A) \neq\emptyset \} .
\]
For $\lambda \leq 1$, a collection $\cC$ of finite subsets of $G$
is said to {\it $\lambda$-cover} a finite subset $A$ of $G$ if
$|A\cap \bigcup \cC | \geq \lambda |A|$. For $\beta \geq 0$, a
collection $\cC$ of finite subsets of $G$ is said to be {\it $\beta$-disjoint}
if for each $C\in\cC$ there is a set $C' \subseteq C$ with $|C' | \geq (1-\beta )|C|$
so that the sets $C'$ for $C\in\cC$ are pairwise disjoint.

\begin{theorem}\label{T-qt}
Let $0<\beta <\frac12$ and let $n$ be a positive integer such that
$(1-\beta /2)^n < \beta$. Then
whenever $e\in T_1 \subseteq T_2 \subseteq \dots \subseteq T_n$
are finite subsets of a group $G$ such that $|\partial_{T_{i-1}} T_i | \leq (\eta /8)|T_i |$ for
$i=2, \dots , n$, for every $(T_n ,\beta /4)$-invariant nonempty finite set $E\subseteq G$
there exist $C_1 , \dots , C_n \subseteq G$ such that
\begin{enumerate}
\item $\bigcup_{i=1}^n T_i C_i \subseteq E$, and

\item the collection of right translates $\bigcup_{i=1}^n \{ T_i c : c\in C_i \}$ is
$\beta$-disjoint and $(1-\beta )$-covers $E$.
\end{enumerate}
\end{theorem}

\begin{theorem}\label{T-Z-stability}
Let $G$ be a countably infinite amenable group and let
$G\curvearrowright X$ be a free minimal action on the Cantor set
such that for every finite set $F\subseteq G$ and $\delta > 0$ there is a
clopen tower decomposition of $X$ with $(F,\delta )$-invariant shapes. Then $C(X)\rtimes G$ is $\cZ$-stable.
\end{theorem}

\begin{proof}
Let $n\in\Nb$. Let $\Upsilon$ be a finite subset of the unit ball of $C(X)$,
$F$ a symmetric finite subset of $G$ containing the identity element $e$,
and $\eps > 0$. Let $a$ be a nonzero positive element of $C(X)\rtimes G$.
We will show the existence of an order-zero completely positive contractive
linear map $\varphi : M_n \to C(X)\rtimes G$ satisfying (i) and (ii) in Theorem~\ref{T-Z-stable}
where the finite set $\Omega$ there is taken to be $\Upsilon\cup \{ u_s : s\in F \}$.
Since $C(X)\rtimes G$ is generated as a C$^*$-algebra by the unit ball of $C(X)$
and the unitaries $u_s$ for $s\in G$, we will thereafter be able to conclude by Theorem~\ref{T-Z-stable}
that $C(X)\rtimes G$ is $\cZ$-stable.

By Lemma~7.9 in \cite{Phi14} we may assume that $a$ is a function in $C(X)$.
Taking a clopen set $A\subseteq X$
on which $a$ is nonzero, we may furthermore assume that $a$ is equal to
the indicator function $\unit_A$.
Minimality implies that the clopen sets $sA$ for $s\in G$ cover $X$,
and so by compactness there is a finite set $D\subseteq G$ such that $D^{-1} A = X$.

Equip $X$ with a compatible metric $d$.
Choose an integer $Q > n^2 / \eps$.

Let $\gamma > 0$, to be determined.
Take a $0 < \beta < 1/n$ which is small enough so that if $T$ is a nonempty finite subset of $G$
which is sufficiently invariant under left translation by $F^Q$ and $T'$ is a subset
of $T$ with $|T' | \geq (1-n\beta )|T|$ then
$|\bigcap_{s\in F^Q} s^{-1} T' | \geq (1-\gamma )|T|$.

Choose an $L\in\Nb$ large enough so that $(1-\beta /2)^L < \beta$.
By amenability there exist finite subsets $e\in T_1 \subseteq T_2 \subseteq\dots\subseteq T_L$ of $G$
such that $|\partial_{T_{l-1}} T_l | \leq (\beta /8)|T_l |$ for $l=2,\dots , L$.
By the previous paragraph, we may also assume that for each $l$ the set $T_l$ is
sufficiently invariant under left translation by $F^Q$ so that
for all $T \subseteq T_l$ satisfying $|T| \geq (1-n\beta )|T_l |$ one has
\begin{align}\label{E-approx invariant}
\bigg|\bigcap_{s\in F^Q} s^{-1} T \bigg| \geq (1-\gamma )|T_l | .
\end{align}

By uniform continuity there is a $\eta > 0$ such that
$|f(x) - f(y)| < \eps /(3n^2 )$ for all $f\in\Upsilon\cup\Upsilon^2$ and all $x,y\in X$
satisfying $d(x,y) < \eta$.
Again by uniform continuity there is an $\eta' > 0$ such that $d(tx,ty) < \eta$ for all
$x,y\in X$ satisfying $d(x,y)<\eta'$ and all $t\in\bigcup_{l=1}^L T_l$.
Fix a clopen partition $\{ A_1 , \dots , A_M \}$ of $X$ whose members all have diameter
less that $\eta'$.

Let $E$ be a finite subset of $G$ containing $T_L$ and let $\delta > 0$ be such that $\delta \leq \beta /4$.
We will further specify $E$ and $\delta$ below.
By hypothesis there is a collection $\{ (V_k , S_k ) \}_{k=1}^K$
of clopen towers such that the shapes $S_1 , \dots , S_K$ are $(E ,\delta )$-invariant
and the sets $S_1 V_1 , \dots , S_K V_K$ partition $X$.
We may assume that for each $k=1,\dots ,K$ the set $S_k$ is large enough so that
\begin{align}\label{E-n}
Mn \bigg( \sum_{l=1}^L |T_l| \bigg) \leq \beta |S_k | .
\end{align}
By a simple procedure we can construct, for each $k$, a clopen partition
$\sP_k$ of $V_k$ such that each level of every one of the towers $(V, S_k )$ for $V\in\sP_k$ is
contained in one of the sets $A_1 , \dots , A_M$ as well as
in one of the sets $A$ and $X\setminus A$. By replacing $(V_k ,S_k )$ with these thinner towers
for each $k$, we may therefore assume that each level in every one of
the towers $(V_1 , S_1 ),\dots ,(V_K , S_K )$ is contained in one of the sets $A_1 , \dots , A_M$
and in one of the sets $A$ and $X\setminus A$.

Let $1\leq k\leq K$. Since $S_k$ is $(T_L ,\beta /4)$-invariant,
by Theorem~\ref{T-qt} and our choice of the sets $T_1 , \dots , T_L$ we can find
$C_{k,1} ,\dots , C_{k,L} \subseteq S_k$ such that the collection
$\{ T_l c : l=1,\dots ,L, \, c\in C_{k,l} \}$ is $\beta$-disjoint, has union contained in $S_k$,
and $(1-\beta )$-covers $S_k$. By $\beta$-disjointness,
for every $l=1,\dots ,L$ and $c\in C_{k,l}$ we can find a $T_{k,l,c} \subseteq T_l$
satisfying $|T_{k,l,c} | \geq (1-\beta )|T_l |$
so that the collection of sets $T_{k,l,c} c$
for $l=1,\dots ,L$ and $c\in C_{k,l}$ is disjoint and has the same union
as the sets $T_l c$ for $l=1,\dots ,L$ and $c\in C_{k,l}$, so that it
$(1-\beta )$-covers $S_k$.

For each $l=1,\dots ,L$ and $m=1,\dots , M$ write $C_{k,l,m}$ for the set of all $c\in C_{k,l}$
such that $cV_k \subseteq A_m$, and choose pairwise disjoint subsets
$C_{k,l,m}^{(1)} , \dots , C_{k,l,m}^{(n)}$ of $C_{k,l,m}$ such that each has cardinality
$\lfloor |C_{k,l,m} | /n \rfloor$.
For each $i=2,\dots , n$ choose a bijection
\begin{align*}
\Lambda_{k,i} : \bigsqcup_{l,m} C_{k,l,m}^{(1)} \to \bigsqcup_{l,m} C_{k,l,m}^{(i)}
\end{align*}
which sends $C_{k,l,m}^{(1)}$ to $C_{k,l,m}^{(i)}$ for all $l,m$.
Also, define $\Lambda_{k,1}$ to be the identity map from $\bigsqcup_{l,m} C_{k,l,m}^{(1)}$ to itself.

Let $1\leq l\leq L$ and $c\in \bigsqcup_m C_{k,l,m}^{(1)}$. Define the set
$T_{k,l,c}' = \bigcap_{i=1}^n T_{k,l,\Lambda_{k,i} (c)}$, which satisfies
\begin{align}\label{E-n beta}
|T_{k,l,c}' | \geq (1-n\beta )|T_l | \geq (1-n\beta )|T_{k,l,c} |
\end{align}
since each $T_{k,l,\Lambda_{k,i} (c)}$ is a subset of $T_l$ of cardinality at least $(1-\beta )|T_l |$.
Set
\begin{align*}
B_{k,l,c,Q} = \bigcap_{s\in F^Q} sT_{k,l,c}' ,\hspace*{5mm}
B_{k,l,c,0} = T_{k,l,c}' \setminus F^{Q-1} B_{k,l,c,Q},
\end{align*}
and, for $q=1,\dots, Q-1$, using the convention $F^0 = \{ e \}$,
\begin{align*}
B_{k,l,c,q} = F^{Q-q} B_{k,l,c,Q} \setminus F^{Q-q-1} B_{k,l,c,Q} .
\end{align*}
Then the sets $B_{k,l,c,0} , \dots , B_{k,l,c,Q}$ partition $T_{k,l,c}'$.
For $s\in F$ we have
\begin{gather}\label{E-B_Q}
sB_{k,l,c,Q} \subseteq B_{k,l,c,Q-1} \cup B_{k,l,c,Q} ,
\end{gather}
while for $q=1,\dots , Q-1$ we have
\begin{align}\label{E-B_q}
sB_{k,l,c,q} \subseteq B_{k,l,c,q-1} \cup B_{k,l,c,q} \cup B_{k,l,c,q+1} ,
\end{align}
for if we are given a $t\in B_{k,l,c,q}$ then $st\in F^{Q-q+1} B_{k,l,c,Q}$, while if
$st\in F^{Q-q-2} B_{k,l,c,Q}$ then $t\in F^{Q-q-1} B_{k,l,c,Q}$ since $F$ is symmetric,
contradicting the membership of $t$ in $B_{k,l,c,q}$.
Also, from (\ref{E-approx invariant}) and (\ref{E-n beta}) we get
\begin{align}\label{E-B approx invariant}
| B_{k,l,c,Q} | \geq (1-\gamma )|T_l | .
\end{align}
For $i=2,\dots ,n$, $c\in \bigsqcup_m C_{k,l,m}^{(i)}$, and $q=0,\dots , Q$ we set
$B_{k,l,c,q} = B_{k,l,\lambda_{k,i}^{-1} (c),q}$.

Write $\Lambda_{k,i,j}$ for the composition $\Lambda_{k,i} \circ\Lambda_{k,j}^{-1}$.
Define a linear map $\psi : M_n \to C(X)\rtimes G$ by declaring it on the standard
matrix units $\{ e_{ij} \}_{i,j=1}^n$ of $M_n$ to be given by
\begin{align*}
\psi (e_{ij} ) = \sum_{k,l,m}
\sum_{c\in C_{k,l,m}^{(j)}} \sum_{t\in T_{k,l,c}'}
u_{t\Lambda_{k,i,j} (c)c^{-1} t^{-1}} \unit_{tcV_k}
\end{align*}
and extending linearly. Then $\psi (e_{ij} )^* = \psi (e_{ji} )$ for all $i,j$ and
the product $\psi (e_{ij} )\psi (e_{i' j'} )$ is $1$ or $0$
depending on whether $i=i'$, so that $\psi$ is a $^*$-homomorphism.

For all $k$ and $l$, all $1\leq i,j\leq n$, and all $c\in \bigsqcup_m C_{k,l,m}^{(j)}$ we set
\begin{align*}
h_{k,l,c,i,j} = \sum_{q=1}^Q \sum_{t\in B_{k,l,c,q}}
\frac{q}{Q} u_{t\Lambda_{k,i,j} (c)c^{-1} t^{-1}} \unit_{tcV_k}
\end{align*}
and put
\begin{align*}
h = \sum_{k,l,m} \, \sum_{i=1}^n \sum_{c\in C_{k,l,m}^{(i)}} h_{k,l,c,i,i} .
\end{align*}
Then $h$ is a norm-one function which commutes with the
image of $\psi$, and so we can define an order-zero completely positive contractive linear map
$\varphi : M_n \to C(X)\rtimes G$ by setting
\begin{align*}
\varphi (z) = h\psi (z) .
\end{align*}
Note that
$\varphi (e_{ij} ) = \sum_{k,l,m} \sum_{c\in C_{k,l,m}^{(j)}} h_{k,l,c,i,j}$.

We now verify condition (ii) in Theorem~\ref{T-Z-stable} for the elements of the set
$\{ u_s : s\in F \}$. Let $1\leq i,j\leq n$.
For all $k$ and $l$, all $c\in \bigsqcup_m C_{k,l,m}^{(j)}$, and all $s\in F$ we have
\begin{align*}
u_s h_{k,l,c,i,j} u_s^{-1} - h_{k,l,c,i,j}
&= \sum_{q=1}^Q \sum_{t\in B_{k,l,c,q}}
\frac{q}{Q} u_{st\Lambda_{k,i,j} (c)c^{-1} (st)^{-1}} \unit_{stcV_k} \\
&\hspace*{20mm} \ - \sum_{q=1}^Q \sum_{t\in B_{k,l,c,q}}
\frac{q}{Q} u_{t\Lambda_{k,i,j} (c)c^{-1} t^{-1}} \unit_{tcV_k} ,
\end{align*}
and so in view of (\ref{E-B_Q}) and (\ref{E-B_q}) we obtain
\begin{align*}
\| u_s h_{k,l,c,i,j} u_s^{-1} - h_{k,l,c,i,j} \| \leq \frac{1}{Q} < \frac{\eps}{n^2} .
\end{align*}
Since each of the elements $b = u_s h_{k,l,c,i,j} u_s^{-1} - h_{k,l,c,i,j}$ is such that
$b^* b$ and $bb^*$ are dominated by twice the indicator functions of
$T_{k,l,\Lambda_j^{-1} (c)}'cV_k$ and
$T_{k,l,\Lambda_j^{-1}(c)}' \Lambda_{k,i,j} (c)V_k$, respectively,
and the sets $T_{k,l,\Lambda_{i'}^{-1}(c)}'cV_k$ over all $k,l$, all $i'=1,\dots ,n$, and all
$c\in \bigsqcup_m C_{k,l,m}^{(i')}$
are pairwise disjoint, this yields
\begin{align*}
\| u_s \varphi (e_{ij} ) u_s^{-1} - \varphi (e_{ij} ) \|
= \max_{k,l,m} \,\max_{c\in C_{k,l,m}^{(j)}} \| u_s h_{k,l,c,i,j} u_s^{-1} - h_{k,l,c,i,j} \| < \frac{\eps}{n^2}
\end{align*}
and hence, for every norm-one element $z = (z_{ij} ) \in M_n$,
\begin{align*}
\| [ u_s ,\varphi (z)] \|
= \| u_s \varphi (z) u_s^{-1} - \varphi (z) \|
&\leq \sum_{i,j=1}^n |z_{ij} | \| u_s \varphi (e_{ij} ) u_s^{-1} - \varphi (e_{ij} ) \| \\
&< n^2 \cdot \frac{\eps}{n^2} = \eps .
\end{align*}

Next we verify condition (ii) in Theorem~\ref{T-Z-stable} for the functions in $\Upsilon$.
Let $1\leq i,j\leq n$.
Let $g\in\Upsilon\cup\Upsilon^2$. Let $1\leq k\leq K$, $1\leq l\leq L$, $1\leq m\leq M$,
and $c\in C_{k,l,m}^{(j)}$.
Then
\begin{align}\label{E-two}
h_{k,l,c,i,j}^* gh_{k,l,c,i,j}
&= \sum_{q=1}^Q \sum_{t\in B_{k,l,c,q}}
\frac{q^2}{Q^2} (tc\Lambda_{k,i,j} (c)^{-1} t^{-1} g) \unit_{tcV_k}
\end{align}
and
\begin{align}\label{E-one}
gh_{k,l,c,i,j}^* h_{k,l,c,i,j}
&= \sum_{q=1}^Q \sum_{t\in B_{k,l,c,q}}
\frac{q^2}{Q^2} g\unit_{tcV_k} .
\end{align}
Now let $x\in V_k$. Since $\Lambda_{k,i,j} (c)x$ and $cx$ both belong to $A_m$,
we have $d(\Lambda_{k,i,j} (c)x,cx) < \eta'$. It follows that for every $t\in T_l$ we have
$d(t\Lambda_{k,i,j} (c)x,tcx) < \eta$ by our choice of $\eta'$, so that
$|g(t\Lambda_{k,i,j} (c)x) - g(tcx)| < \eps /(3n^2 )$
by our choice of $\eta$, in which case
\begin{align*}
\| (tc\Lambda_{k,i,j} (c)^{-1} t^{-1} g - g)\unit_{tcV_k} \|
&= \| c^{-1} t^{-1} ((tc\Lambda_{k,i,j} (c)^{-1} t^{-1} g - g)\unit_{tcV_k}) \| \\
&= \| (\Lambda_{k,i,j} (c)^{-1} t^{-1} g - c^{-1} t^{-1} g)\unit_{V_k} \| \\
&= \sup_{x\in V_k} |g(t\Lambda_{k,i,j} (c)x) - g(tcx)| \\
&< \frac{\eps}{3n^2} .
\end{align*}
Using (\ref{E-two}) and (\ref{E-one}) this gives us
\begin{align}\label{E-max}
\lefteqn{\| h_{k,l,c,i,j}^* gh_{k,l,c,i,j} - gh_{k,l,c,i,j}^* h_{k,l,c,i,j} \|}\hspace*{15mm} \\
\hspace*{10mm} &= \max_{q=1,\dots ,Q} \,\max_{t\in B_{k,l,c,q}} \frac{q^2}{Q^2}
\| (tc\Lambda_{k,i,j} (c)^{-1} t^{-1} g - g)\unit_{tcV_k} \|
< \frac{\eps}{3n^2} .\notag
\end{align}
Set $w = \varphi (e_{ij} )$ for brevity. Let $f\in\Upsilon$.
For $g \in \{ f , f^2 \}$ the functions
$h_{k,l,c,i,j}^* gh_{k,l,c,i,j} - gh_{k,l,c,i,j}^* h_{k,l,c,i,j}$
over all $k$, $l$, and $m$ and all $c\in C_{k,l,m}^{(j)}$
have pairwise disjoint supports, so that (\ref{E-max}) yields
\begin{align*}
\| w^* gw - gw^* w \| < \frac{\eps}{3n^2} .
\end{align*}
It follows that
\begin{align*}
\| w^* f^2 w - fw^* fw \|
\leq \| w^* f^2 w - f^2 w^* w \| + \| f(fw^* w - w^* fw) \|
< \frac{2\eps}{3n^2}
\end{align*}
and so
\begin{align*}
\| fw - wf \|^2
&= \| (fw-wf)^* (fw - wf) \| \\
&= \| w^* f^2 w - fw^* fw + fw^* wf - w^* fwf \| \\
&\leq \| w^* f^2 w - fw^* fw \| + \| (fw^* w - w^* fw)f \| \\
&< \frac{2\eps}{3n^2} + \frac{\eps}{3n^2} = \frac{\eps}{n^2}.
\end{align*}
Therefore for every norm-one element $z = (z_{ij} ) \in M_n$ we have
\begin{align*}
\| [f,\varphi (z)] \|
\leq \sum_{i,j=1}^n |z_{ij} | \| [f,\varphi (e_{ij} )] \|
< n^2 \cdot \frac{\eps}{n^2} = \eps .
\end{align*}

Finally, we verify that the parameters in the construction of $\varphi$
can be chosen so that $1-\varphi (1) \precsim \unit_A$.
By taking the sets $S_1 , \dots , S_K$ to be
sufficiently left invariant (by enlarging $E$ and shrinking $\delta$ if necessary)
we may assume that for every $k=1,\dots ,K$ there is an $S_k' \subseteq S_k$
such that the set $\{ s\in S_k' : Ds\subseteq S_k \}$ has cardinality
at least $|S_k|/2$. Let $1\leq k\leq K$. Take a maximal set $S_k'' \subseteq S_k'$
such that the sets $Ds$ for $s\in S_k''$ are pairwise disjoint, and note that
$|S_k'' |\geq |S_k' |/|D^{-1} D| \geq |S_k |/(2|D|^2 )$.
Since $D^{-1} A = X$, each of the sets $DsV_k$ for $s\in S_k''$ intersects $A$,
and so the set $S_k^\sharp$ of all $s\in S_k$ such that $sV_k \subseteq A$
has cardinality at least $|S_k|/(2|D|^2 )$.
Define $S_{k,1} = \bigsqcup_{l,m} \bigsqcup_{i=1}^n \bigsqcup_{c\in C_{k,l,m}^{(i)}} B_{k,l,c,Q} c$,
which is the set of all $s\in S_k$ such that the function $\varphi (1)$ takes the value $1$ on $sV_k$.
Set $S_{k,0} = S_k \setminus S_{k,1}$.
Since $\sum_{i=1}^n |C_{k,l,m}^{(i)} | \geq |C_{k,l,m}| - n$ for every $l$ and $m$,
by (\ref{E-n}) we have
\begin{align*}
\sum_{l,m} \sum_{i=1}^n |T_l ||C_{k,l,m}^{(i)}|
&\geq \sum_{l,m} |T_l ||C_{k,l,m}| - Mn \sum_l |T_l | \\
&\geq \bigg| \bigcup_l T_l C_{k,l} \bigg| - \beta |S_k |\\
&\geq (1-2\beta ) |S_k | .
\end{align*}
Since for all $l$ and $i$ and all $c\in C_{k,l,m}^{(i)}$
we have $|B_{k,l,c,Q} | \geq (1-\gamma )|T_l|$ by (\ref{E-B approx invariant}),
it follows, putting $\lambda = (1-\gamma )(1-2\beta )$, that
\begin{align*}
|S_{k,1} |
\geq (1-\gamma ) \sum_{l,m} \sum_{i=1}^n |T_l ||C_{k,l,m}^{(i)}|\geq \lambda |S_k | .
\end{align*}
By taking $\gamma$ and $\beta$ small enough we can guarantee that
$1-\lambda \leq 1/(2|D|^2 )$ and hence
\begin{align*}
|S_{k,0} | = |S_k | - |S_{k,1} | \leq (1-\lambda )|S_k | \leq |S_k^\sharp | ,
\end{align*}
so that there exists an injection $\theta_k : S_{k,0} \to S_k^\sharp$.
Define
\[
z = \sum_{k=1}^K \sum_{s\in S_{k,0}} u_{\theta_k (s)s^{-1}} \unit_{sV_k} .
\]
A simple computation shows that $z^* \unit_A z$ is
the indicator function of $\bigsqcup_{k=1}^K S_{k,0} V_k$, which is the support of $1-\varphi (1)$,
and so putting $v = (1-\varphi (1))^{1/2} z^*$ we get
\begin{align*}
v\unit_A v^* = (1-\varphi (1))^{1/2}z^* \unit_A z (1-\varphi (1))^{1/2} = 1-\varphi (1) .
\end{align*}
This demonstrates that $1-\varphi (1) \precsim \unit_A$, as desired.
\end{proof}

Combining Theorems~\ref{T-Z-stable} and \ref{T-dense G delta} yields the following.

\begin{theorem}\label{T-comeager}
Let $G$ be a countably infinite amenable group and $X$ the Cantor set. Then the set of all actions in
$\frmin (G,X)$ whose crossed product is $\cZ$-stable is comeager, and in particular nonempty.
\end{theorem}

\begin{question}\label{Q-nuclear dimension}
Do any of the crossed products in Theorem~\ref{T-comeager} have tracial state space with compact
extreme boundary (from which we would be able to conclude finite nuclear dimension by \cite{BosBroSatTikWhiWin16}
and hence classifiability)?
For $G=\Zb$ a generic action in $\frmin (G,X)$
is uniquely ergodic, so that the crossed product has a unique tracial state \cite{Hoc08}.
However, already for $\Zb^2$ nothing of this nature seems to be known.
On the other hand, it is known that the crossed products of free minimal actions of finitely generated
nilpotent groups on compact metrizable spaces of finite covering dimension
have finite nuclear dimension, and in particular are $\cZ$-stable \cite{SzaWuZac14}.
\end{question}


\begin{thebibliography}{999}

\bibitem{BosBroSatTikWhiWin16}
J. Bosa, N. Brown, Y. Sato, A. Tikuisis, S. White and W. Winter.
Covering dimension of C$^*$-algebras and 2-coloured classification.
To appear in {\it Mem.\ Amer.\ Math.\ Soc.}.

\bibitem{CecGriHar99}
T. Ceccherini-Silberstein, P. de la Harpe, and R. I. Grigorchuk.
Amenability and paradoxical decompositions for pseudogroups and discrete metric spaces. (Russian)
{\it Tr.\ Mat.\ Inst.\ Steklova} {\bf 224} (1999), 68--111;
translation in {\it Proc.\ Steklov Inst.\ Math.} {\bf 224} (1999), 57--97.

\bibitem{DowHucZha16}
T. Downarowicz, D. Huczek, and G. Zhang. Tilings of amenable groups.
To appear in {\it J. Reine Angew.\ Math.}

\bibitem{EllGonLinNiu15}
G. Elliott, G. Gong, H. Lin, and Z. Niu.
On the classification of simple amenable C$^*$-algebras with finite decomposition rank, II.
arXiv:1507.03437.

\bibitem{EleLip10}
G. Elek and G. Lippner.
Borel oracles. An analytic approach to constant time algorithms.
{Proc.\ Amer.\ Math.\ Soc.} {\bf 138} (2010), 2939--2947.

\bibitem{GonLinNiu15}
G. Gong, H. Lin, and Z. Niu.
Classification of finite simple amenable $\cZ$-stable C$^*$-algebras.
arXiv:1501.00135.

\bibitem{GueWilYu16}
E. Guentner, R. Willett, and G. Yu. Dynamic asymptotic dimension: relation to dynamics,
topology, coarse geometry, and C$^*$-algebras.
{\it Math.\ Ann.}\ {\bf 367} (2017), 785--829.

\bibitem{HirOro13}
I. Hirshberg and J. Orovitz.
Tracially $\cZ$--absorbing C$^*$-algebras.
{\it J. Funct.\ Anal.} {\bf 265} (2013), 765--785.

\bibitem{Hoc08}
M. Hochman.
Genericity in topological dynamics.
{\it Ergodic Theory Dynam.\ Systems} {\bf 28} (2008), 125--165.

\bibitem{JiaSu99}
X. Jiang and H. Su.
On a simple unital projectionless C$^*$-algebra.
{\it Amer.\ J. Math.} {\bf 121} (1999), 359--413.

\bibitem{KerLi17}
D. Kerr and H. Li. {\it Ergodic Theory: Independence and Dichotomies}.
Springer, Cham, 2016.

\bibitem{KecSolTod99}
A. Kechris, S. Solecki, and S. Todorcevic.
Borel chromatic numbers, {\it Adv.\ in Math.} {\bf 141} (1999), 1--44.

\bibitem{LyoNaz11}
R. Lyons and F. Nazarov.
Perfect matchings as IID factors on non-amenable groups.
{\it European J. Combin.} {\bf 32} (2011), 1115--1125.

\bibitem{MatSat12}
H. Matui and Y. Sato.
Strict comparison and $\cZ$-absorption of nuclear C$^*$-algebras.
{\it Acta Math.} {\bf 209} (2012), 179--196.

\bibitem{MatSat14}
H. Matui and Y. Sato.
Decomposition rank of UHF-absorbing C$^*$-algebras.
{\it Duke Math. J.} {\bf 163} (2014), 2687--2708.

\bibitem{Nam64}
I. Namioka,
F{\o}lner's conditions for amenable semi-groups. {\it Math.\ Scand.} {\bf 15} (1964), 18--28.

\bibitem{OrnWei87}
D. S. Ornstein and B. Weiss. Entropy and isomorphism theorems for actions of amenable groups.
{\it J. Analyse Math.} {\bf 48} (1987), 1--141.

\bibitem{Phi14}
N. C. Phillips. Large subalgebras. arXiv:1408.5546.

\bibitem{RorWin10}
M. R{\o}rdam and W. Winter. The Jiang--Su algebra revisited.
{\it J. Reine Angew.\ Math.} {\bf 642} (2010), 129--155.

\bibitem{Sza15}
G. Szab{\'o}.
The Rokhlin dimension of topological $\Zb^m$-actions.
{\it Proc.\ Lond.\ Math.\ Soc.\ (3)} {\bf 110} (2015), 673--694.

\bibitem{SzaWuZac14}
G. Szabo, J. Wu, and J. Zacharias.
Rokhlin dimension for actions of residually finite groups.
arXiv:1408.6096.

\bibitem{TikWhiWin17}
A. Tikuisis, S. White, and W. Winter.
Quasidiagonality of nuclear C$^*$-algebras. {\it Ann.\ of Math.\ (2)} {\bf 185} (2017), 229--284.

\bibitem{Win12}
W. Winter. Nuclear dimension and $\cZ$-stability of pure C$^*$-algebras.
{\it Invent.\ Math.}\ {\bf 187} (2012), 259--342.

\bibitem{WinZac09}
W. Winter and J. Zacharias.
Completely positive maps of order zero.
{\it M\"{u}nster J. Math.}\ {\bf 2} (2009), 311--324.
\end{thebibliography}
\end{document}